\documentclass[12pt,reqno]{amsart}

\usepackage[utf8]{inputenc}
\usepackage[T1]{fontenc}

\usepackage[usenames,dvipsnames]{color}

\usepackage{amsfonts, amsmath, amssymb, amscd, latexsym, amsthm}

\usepackage[english]{babel}

\usepackage{enumerate}
\usepackage{longtable}
\usepackage{multirow,longtable,array}
\usepackage{geometry}
\usepackage{float}
\usepackage{adjustbox}
\usepackage{rotating}
\usepackage{booktabs}
\usepackage{array}

\usepackage{tikz}
\usetikzlibrary{shapes,arrows}
\usepackage{xypic}
\usepackage{pifont}

\usepackage[pagebackref]{hyperref}
\renewcommand*{\backref}[1]{}
\renewcommand*{\backrefalt}[4]{%
	\ifcase #1
	(\tt not cited)%
	\or
	(\tt cited on page~#2)%
	\else
	(\tt cited on pages~#2)%
	\fi}
\usepackage{hyperref}
\hypersetup{
	colorlinks = true,
	urlcolor   = blue,
	citecolor  = blue
}

\newtheorem{theorem}{Theorem}[section]
\newtheorem{lemma}[theorem]{Lemma}

\theoremstyle{remark}

\def\QQ{\mathbb{Q}}

\def\ZZ{\mathbb{Z}}

\usepackage{dsfont, amsfonts, amsmath, amssymb,amscd, stmaryrd, latexsym, amsthm, dsfont}
\usepackage{enumerate}
\usepackage{longtable}
\usepackage{geometry}
\usepackage{float}
\usepackage{tikz}
\usetikzlibrary{shapes,arrows}

\usepackage{pifont}
\usepackage{float}
\usepackage{tikz}
\usepackage{xypic}
\usetikzlibrary{shapes,arrows}
\usepackage{multirow}
\usepackage{multicol}
\begin{document}
	
	\def\2r{\mathrm{rank}}
	\def\rank{\mathrm{rank}}
	\def\rg{\mathrm{rank}}
	\def\C{\mathrm{C}}
	\def\Y{\mathbf{Y}}
	\def\V{\mathbf{V}}
	\def\vep{\varepsilon}
	
	\selectlanguage{english}
	
	\title[On real biquadratic number fields for which the  $2$-Iwasawa module  has rank $3$]{On real biquadratic number fields for which the  $2$-Iwasawa module  has rank $3$}

	\author[M. M. Chems-Eddin]{Mohamed Mahmoud Chems-Eddin}
	\address{Mohamed Mahmoud CHEMS-EDDIN: Department of Mathematics, Faculty of Sciences Dhar El Mahraz, Sidi Mohamed Ben Abdellah University, Fez,  Morocco}
	\email{2m.chemseddin@gmail.com}
	
	\author[I. Daqaq]{Ikrame Daqaq}
	\address{Ikrame DAQAQ: Departement of Mathematics, Faculty of Sciences Dhar El Mahraz,  Sidi Mohamed Ben Abdellah University, Fez, Morocco}
	\email{ikramedq@gmail.com}
	
	
	\begin{abstract} 
		We investigate  the stability of the $2$-rank in the cyclotomic $\mathbb{Z}_2$-extension of real biquadratic number fields. We determine several families of real biquadratic fields $F$ for which the ranks of the $2$-class group $A(F)$ and the associated $2$-Iwasawa module $A_\infty(F)$   are both equal to $3$. More precisely, we provide a complete classification of all fields $F=\mathbb{Q}(\sqrt{q},\sqrt{d}),$ where $q\equiv 3\pmod{4}$ is a prime number and $d$ is an odd positive square-free integer, and all fields $F=\mathbb{Q}(\sqrt{2q},\sqrt{d}),$ 
	where $q\equiv 3\pmod{4}$ is a prime number and $d\equiv 1\pmod{4}$ is a positive square-free integer, for which
	$	\operatorname{rank}(A(F))=\operatorname{rank}(A_\infty(F))=3.$
	 	\end{abstract}
	
	\maketitle
	\section{\bf Introduction}
	
Let $\ell$ denote   a prime number. Let $\ZZ_\ell$  be the additive group of $\ell$-adic integers,  $\Lambda=\mathbb Z_\ell \llbracket  T \rrbracket$, the formal series ring, and 	$M$  a finitely generated   $\Lambda$-module.   It is known, by the structure theorem,  that  there exist  $r$, $s$,  $t$, $n_i$, $m_i  \in\mathbb Z$,  and  $f_i$   distinguished and irreducible polynomials  of  $\mathbb Z_\ell[T]$ such that 
\begin{eqnarray*}
	M\sim \Lambda^r \oplus \left( \bigoplus_{i=1}^{s} \Lambda/ (\ell^{n_i})\right)\oplus \left(  \bigoplus_{j=1}^{t}\Lambda/(f_j(T)^{m_j}) \right).
\end{eqnarray*}

Let  $k$ be a   number field and $\ell$ a prime number. Let $A_\ell(k)$  or simply $A(k)$ when $\ell=2$  (resp.   $E_k$) denote the $\ell$-class group (resp.   the unit group) of $k$.
A   $\ZZ_\ell$-extension of $k$ is an infinite extension of $k$ denoted by  $k_\infty$ such that $\mathrm{Gal}(k_\infty/k)\simeq \ZZ_\ell$. For each $n\geq 1$, the extension $k_\infty/k$ contains a unique field denoted by $k_n$ and called the $n$-th layer, which has degree $\ell^n$ over $k$, of the cyclotomic $\ZZ_\ell$-extension of $k$. Furthermore, we have:
$$k=k_0\subset k_1 \subset k_2 \subset\cdots \subset k_n \subset \cdots  \subset k_\infty=\bigcup_{n\geq 0} k_n.$$
In particular, for an odd prime number $\ell$, let $ \QQ_{\ell, n}$   be the unique real subfield  of the cyclotomic field  $\QQ(\zeta_{\ell^{n+1}}) $ of degree $\ell^n$ over $\QQ$,  
and for $\ell=2$,   let $  \QQ_{2, n}$ be the field $\QQ(2\cos( {2\pi}/{2^{n+2}}))$, for all $n\geq 1$.
Then $k_\infty= \bigcup_{n\geq 0} k_n$, where $k_n= k\QQ_{\ell, n}$,  is called the   cyclotomic $\ZZ_\ell$-extension of $k$.  
The inverse limit $A(k_\infty):=\varprojlim A_\ell(k_n)$
with respect to the norm maps is called the Iwasawa module for $k_\infty/k$; and it is   a finitely generated $\Lambda$-module (cf. \cite{iwasawa59} and \cite{washington1997introduction} for more details).


	For $\ell=2$, several authors have recently investigated the $2$-rank of $A_\infty(k)$ when $k$ is a quadratic or biquadratic number field (see, e.g., \cite{chemskatharina,Elmahi1,Elmahi2,Elmahi3,8,Laxmi,mizusawa5,MR2024}). It is worth noting that all the results in these papers concerning the stability of the $2$-rank of the class group along the cyclotomic $\mathbb{Z}_2$-extension of the given number field are restricted to the case where this $2$-rank is at most $2$; that is,
	\[
	\operatorname{rank}(A(k))=\operatorname{rank}(A_\infty(k))\leq 2
	\]
	for all $n\geq 1$. An exception is provided by Mizusawa in \cite{mizusawa5}, who constructed infinite families of real quadratic number fields $k$ satisfying
	\[
	\operatorname{rank}(A(k_1))=\operatorname{rank}(A_\infty(k))=3.
	\]

Motivated by these results on real quadratic number fields, we construct several families of	real biquadratic number fields $F$
such that 
\begin{eqnarray}
	 \operatorname{rank}(A(F))=\operatorname{rank}(A_\infty(F))=3.
\end{eqnarray}
  	More precisely, letting $F= \mathbb{Q}(\sqrt{q},\sqrt{d})$, with $q\equiv3\pmod 4$ a prime number and $d$ an odd positive square-free integer, or 
  	 $F= \mathbb{Q}(\sqrt{2q},\sqrt{d})$, with  $q\equiv 3 \pmod{4}$ a prime number  and $d\equiv 1 \pmod{4}$   a positive square-free integer, we give the complete classification of all these fields   such that $\operatorname{rank}(A(F))=\operatorname{rank}(A_\infty(F))=3$.
		
Letting $L:=\mathbb{Q}(\sqrt{q},\sqrt{r}),$
	where $r\equiv 1 \pmod{8}\quad\text{ and }
	\left(\frac{q}{r}\right)=1$, we state the main theorem of this paper as follows:
\begin{theorem}[\bf The Main Theorem] 
 Let $F$ be a real biquadratic number field such that  $F \neq L$ and $F= \mathbb{Q}(\sqrt{q},\sqrt{d})$, with $q\equiv3\pmod 4$ a prime number and $d$ an odd positive square-free integer, or 
 	$F= \mathbb{Q}(\sqrt{2q},\sqrt{d})$, with  $q\equiv 3 \pmod{4}$ a prime number and $d\equiv 1 \pmod{4}$ a positive square-free integer. Then $\operatorname{rank}(A(F_{\infty}))=3$ and $\operatorname{rank}(A(F_{\infty}))=\operatorname{rank}(A(F))$ if and only if $F$ takes one of the following forms:
	\begin{enumerate}[$1)$]

		\item 	 $F=\QQ(\sqrt{q},\sqrt{d})$, 
		where $d>2$ is  relatively prime to $q$, and
		$d=rst$, for prime numbers $r$, $s$ and $t$  such that  $ \left( \frac{q}{s}\right)=\left( \frac{q}{r}\right)=-\left( \frac{q}{t}\right)=1 $ and one of the following conditions holds:
		\begin{enumerate}[$a)$]
			\item  	$r\equiv s \equiv 3 \pmod 8$, and $[t \equiv 5 \pmod 8 \text{ or } (t \equiv 3 \pmod 8 \text{ and }q \equiv 7 \pmod 8) \text{ or } (t \equiv 7 \pmod 8 \text{ and }q \equiv 3 \pmod 8)]$,
			\item $r\equiv s \equiv 5 \pmod 8$, and $ [t \equiv 3 \pmod 8 \text{ or } t \equiv 5 \pmod 8 \text{ or }  (t \equiv 7 \pmod 8 \text{ and } q \equiv 3 \pmod 8 )]$.

		\end{enumerate}
		
		\item  $F=\QQ(\sqrt{q},\sqrt{d})$, where
		  $d>2$ is relatively prime to $q$, and
		$d=rst$, for prime numbers $r$, $s$ and $t$   such that $ \left( \frac{q}{s}\right)=\left( \frac{q}{r}\right)=\left( \frac{q}{t}\right)=1 $ and one of the following conditions holds:
	 	\begin{enumerate}[$a)$]
			\item $q \equiv 3 \pmod 8$,  $r\equiv t\equiv 3 \pmod 8$ and $s \equiv 5 \pmod 8$,
			\item  $q \equiv 3 \pmod 8$, $r \equiv 3\pmod 8$ and $s \equiv t \equiv 5 \pmod 8$,

			\item  $q \equiv 3 \pmod 8$, $r \equiv 3\pmod 8$, $s \equiv 5 \pmod 8$, $ t \equiv 7 \pmod 8$ and $\overline{\sqrt{\varepsilon_{2q}}} \notin \langle \overline{-1}, \overline{\varepsilon_2}, \overline{\sqrt{\varepsilon_{q}}} \rangle$, 
			
			\item  $q \equiv 3 \pmod 8$, $r \equiv t \equiv 5\pmod 8$, $s \equiv  7 \pmod 8$ and  $\overline{\sqrt{\varepsilon_{2q}}}   \notin \langle \overline{-1}, \overline{\varepsilon_2}, \overline{\sqrt{\varepsilon_{q}}} \rangle$,
			
			\item $q \equiv 7 \pmod 8$, $r \equiv t \equiv 3 \pmod 8$  and $s  \equiv 5 \pmod 8$,
			
			\item $q \equiv 7 \pmod 8$, $r  \equiv 3 \pmod 8$  and $s \equiv t  \equiv 5 \pmod 8$,
			
			\item  $q \equiv 3 \pmod 8$, $r\equiv 3 \pmod 8$, $s\equiv 7 \pmod 8$, $t \equiv 3 \pmod 8$ and such that $\overline{\sqrt{\varepsilon_{2q}}},\overline{\sqrt{\varepsilon_{q}}}  \notin \langle \overline{-1}, \overline{\varepsilon_2} \rangle$.
		\end{enumerate}
		
		\item
	
		$F=\QQ(\sqrt{2q},\sqrt{d})$,   where   $d=\delta rst$,  for prime numbers     $r$, $s$, $t$    and an integer  $\delta \in \{ 1,q \} $, such that one of the following conditions holds: 
		\begin{enumerate}[$a)$]
		 	\item $q \equiv 3 \pmod 8$, $r \equiv s \equiv  3 \pmod 8$, $\left( \frac{q}{r}\right)=\left( \frac{q}{s}\right)=-1$, $t \equiv 7 \pmod 8$ and $\left( \frac{q}{t} \right)=-1$,
			\item $q \equiv 3 \pmod 8$, $r \equiv s \equiv  5 \pmod 8$, $\left( \frac{q}{r}\right)=\left( \frac{q}{s}\right)=-1$, $t \equiv 7 \pmod 8$ and  $\left( \frac{q}{t} \right)=-1$,
			\item \ $q \equiv 3 \pmod 8$, $r \equiv s \equiv  3 \pmod 8$, $\left( \frac{q}{r}\right)=\left( \frac{q}{s}\right)=-1$, $t \equiv 5 \pmod 8$ and  $ \left( \frac{q}{t} \right)=1$,
			\item \ $q \equiv 3 \pmod 8$, $r \equiv s \equiv  5 \pmod 8$, $\left( \frac{q}{r}\right)=\left( \frac{q}{s}\right)=-1$, $t \equiv 3$, $5 \pmod 8$ and  $ \left( \frac{q}{t} \right)=1$,
			\item \ $q \equiv 7 \pmod 8$, $r \equiv s \equiv  3 \pmod 8$, $\left( \frac{q}{r}\right)=\left( \frac{q}{s}\right)=-1$, $t \equiv 3$, $5 \pmod 8$ and  $ \left( \frac{q}{t} \right)=1$,
			\item \ $q \equiv 7 \pmod 8$, $r \equiv s \equiv  5 \pmod 8$, $\left( \frac{q}{r}\right)=\left( \frac{q}{s}\right)=-1$, $t \equiv 3$, $5 \pmod 8$ and  $ \left( \frac{q}{t} \right)=1$,
			\item $q \equiv 3 \pmod 8$, $r \equiv 3 \pmod 8$, $s \equiv 3 \pmod 8$, $\left( \frac{q}{r}\right)=\left( \frac{q}{s}\right)=-1$, $t \equiv 5 \pmod 8$ and $\left( \frac{q}{t}\right)=-1$,
			\item $q \equiv 3 \pmod 8$, $r \equiv 3 \pmod 8$, $s \equiv 3 \pmod 8$, $\left( \frac{q}{r}\right)=\left( \frac{q}{s}\right)=-1$, $t \equiv 7 \pmod 8$,  $\left( \frac{q}{t}\right)=1$ and such that $\overline{\sqrt{\varepsilon_{2q}}},\overline{\sqrt{\varepsilon_{q}}}  \notin \langle \overline{-1}, \overline{\varepsilon_2} \rangle$,
			\item $q \equiv 3 \pmod 8$, $r \equiv 3 \pmod 8$, $s \equiv 5 \pmod 8$, $\left( \frac{q}{r}\right)=\left( \frac{q}{s}\right) =-1$, $t \equiv 5 \pmod 8$ and $\left( \frac{q}{t}\right)=-1$,
			\item $q \equiv 3 \pmod 8$, $r \equiv 3 \pmod 8$, $s \equiv 5 \pmod 8$,   $\left( \frac{q}{r}\right)=\left( \frac{q}{s}\right)=-1$, $t \equiv 7 \pmod 8$, $\left( \frac{q}{t}\right)=1$ and such that $\overline{\sqrt{\varepsilon_{q}}}  \notin \langle \overline{\sqrt{\varepsilon_{2q}}},\overline{-1}, \overline{\varepsilon_2} \rangle$,
			\item $q \equiv 3 \pmod 8$, $r \equiv 7 \pmod 8$, $s \equiv 3 \pmod 8$, $t \equiv 5 \pmod 8$, $\left( \frac{q}{r}\right)=1$, $\left( \frac{q}{s}\right)=\left( \frac{q}{t}\right)=-1$ and such that $\overline{\sqrt{\varepsilon_{q}}}  \notin \langle \overline{\sqrt{\varepsilon_{2q}}},\overline{-1}, \overline{\varepsilon_2} \rangle$,
			\item $q \equiv 3 \pmod 8$, $r \equiv 7 \pmod 8$, $s \equiv 5 \pmod 8$, $t \equiv 5 \pmod 8$, $\left( \frac{q}{r}\right)=1$, $\left( \frac{q}{s}\right)=\left( \frac{q}{t}\right)=-1$ and such that $\overline{\sqrt{\varepsilon_{q}}}  \notin \langle \overline{\sqrt{\varepsilon_{2q}}},\overline{-1}, \overline{\varepsilon_2}\rangle$,
			\item $q \equiv 7 \pmod 8$, $r \equiv 3 \pmod 8$, $s \equiv 3 \pmod 8$, $t \equiv 5 \pmod 8$, $\left( \frac{q}{r}\right)=-1$ and $\left( \frac{q}{s}\right)=\left( \frac{q}{t}\right)=-1$,
			\item $q \equiv 7 \pmod 8$, $r \equiv 3 \pmod 8$, $s \equiv 5 \pmod 8$, $t \equiv 5 \pmod 8$, $\left( \frac{q}{r}\right)=-1$ and $\left( \frac{q}{s}\right)=\left( \frac{q}{t}\right)=-1$.
		\end{enumerate}
		\end{enumerate}
Here, for a unit $u\in F'_{1}:=\QQ(\sqrt{q},\sqrt{2})$,  we let $\overline{u}$ denote the class of $u$ in the quotient	  $E_{F'_{1}} / E_{F'_{1}} \cap N_{F_1/F'_{1}}(k^\times)$,  where $F_{1}:=\QQ(\sqrt{q},\sqrt{d},\sqrt{2} )$.
\end{theorem}

\section{\bf Preliminary results}\label{sec2}

We begin by recalling some preliminary results that will be needed throughout the paper.

\begin{lemma}[\cite{kuruda}]
	Let $k$ be a multiquadratic number field of degree $2^n$, with $n\geq 2$ an integer, and $k_i$, $i=1,\ldots,2^n-1$, the  quadratic subfields of $k$. Then
	$$h(k)=\frac{1}{2^v}q(k)\prod_{i=1}^{s}h(k{_i}),$$
	with $q(k)=[E_{k}: \prod_{i=1}^{s}E_{k_{i}}]$ and $$v=\left\{ \begin{array}{cl}
		n(2^{n-1}-1); &\text{ if } k \text{ is real, }\\
		(n-1)(2^{n-2}-1)+2^{n-1}-1 & \text{ if } k \text{ is imaginary.}
	\end{array}\right.$$
\end{lemma}

\begin{lemma}[{\cite[Lemma 2.4]{yue}}]
	Let $k/k'$ be a  quadratic extension of number fields such that $h_2(k')$ is odd. Then the rank of the $2$-class group of $k$ is given by
	\[
	\operatorname{rank}(A(k))
	=
	t_{k/k'}-1-e_{k/k'},
	\]
	where $t_{k/k'}$ is the number of ramified primes (finite or infinite) in the extension $k/k'$ and $e_{k/k'}$ is defined by
	\[
	2^{e_{k/k'}}
	=
	\left[E_{k'} : E_{k'} \cap N_{k/k'}(k^\times)\right].
	\]
\end{lemma}

	The following lemma is a particular case of Fukuda's Theorem  \cite{fukudaa}.

\begin{lemma}[{\cite{fukudaa}}]
	Let $k_{\infty}/k$ be a $\mathbb{Z}_2$-extension and $n_0$ an integer such that any prime of $k_{\infty}$ which is ramified in $k_{\infty}/k$ is totally ramified in $k_{\infty}/k_{n_0}$.
 	\begin{enumerate}[$1)$]
		\item If there exists an integer $n \geq n_0$ such that $h_2(k_n)=h_2(k_{n+1}),$ then $h_2(k_n)=h_2(k_m)$ 
		for all $m \geq n$.
		
		\item If there exists an integer $n \geq n_0$ such that
		$\operatorname{rank}(A(k_n))
		=\operatorname{rank}(A(k_{n+1})),$
		then $\operatorname{rank}(A(k_m))=\operatorname{rank}(A(k_n))$
		for all $m \geq n$.
	\end{enumerate}
\end{lemma}

\begin{lemma}[\cite{2MCD2026}, Proof of Lemma 3.5 and Remark 3.6]\label{chemseddin} 
 	Let $F'_{1}:=\QQ(\sqrt{q},\sqrt{2})$,  $F_{1}:=\QQ(\sqrt{q},\sqrt{d},\sqrt{2} )$  and  let $r$ be an odd prime number and $\mathcal R$ be a prime ideal of
	$F_1'$ lying above $r$. Let $\gamma=0$ or $1$ according to whether  $q \equiv 7 \pmod 8$ or  $q \equiv 3 \pmod 8$.
 	If $r$ is not totally decomposed in $F_1'$, then $\left(\frac{-1,r}{\mathcal R}\right)=1,$ else we have $\left(\frac{-1,r}{\mathcal R}\right)
	=
	\left(\frac{-1}{r}\right).$
 	Moreover, we have:
	\begin{enumerate}[$1)$]
		\item  	Assume that $r\equiv 3$ or $5$ $ \pmod 8$. Then 
		 \begin{equation*}
			\left(\frac{\varepsilon_2,r}{\mathcal R}\right)
			=
			\left(\frac{-1}{r}\right)
			=
			(-1)^{\frac{r-1}{2}}.
		\end{equation*}
	 	\begin{enumerate}[$\bullet$]
			
			\item If $\left(\frac qr\right)=-1$, 
			then 
			
			\begin{align*}
				\left(\frac{\sqrt{\varepsilon_q},r}{\mathcal R}\right)
				&=
				\left(\frac{(-1)^{\gamma+1}}{r}\right), \\[2mm]
				\left(\frac{\sqrt{\varepsilon_{2q}},r}{\mathcal R}\right)
				&=
				\left(\frac{-\varepsilon_{2q},r}
				{\mathcal R_{\mathbb Q(\sqrt{2q})}}\right)
				=
				\left(\frac{-1}{r}\right)
				\left(\frac{2}{r}\right).
			\end{align*}
			
			\item  If $\left(\frac qr\right)=1$, then 
			\begin{align*}
				\left(\frac{\sqrt{\varepsilon_q},r}{\mathcal{R}}\right)
				&=
				\left(\frac{-\varepsilon_q,r}{\mathcal{R}_{\mathbb Q(\sqrt 2q)}}\right)
				=
				\left(\frac{-1}{r}\right)
				\left(\frac{2}{r}\right),\\[2mm]
				\left(\frac{\sqrt{\varepsilon_{2q}},r}{\mathcal{R}}\right)
				&=
				\left(\frac{(-1)^{\gamma+1}}{r}\right).
			\end{align*}
			
		\end{enumerate}

		\item	Assume that $r\equiv 7$ or $1$ $ \pmod 8$ with $\left(\frac{q}{r}\right)=-1.$ Then 
		
		\begin{align*}
			\left(\frac{\varepsilon_2,r}{\mathcal{R}}\right)
			&=1,\\[2mm]
			\left(\frac{\sqrt{\varepsilon_q},r}{\mathcal{R}}\right)
			&=\left(\frac{(-1)^{\gamma}}{r}\right),\\[2mm]
			\left(\frac{\sqrt{\varepsilon_{2q}},r}{\mathcal{R}}\right)
			&=\left(\frac{(-1)^{\gamma}}{r}\right).
		\end{align*}

		\item	Assume that $r\equiv 7$ $ \pmod 8$ with $\left(\frac{q}{r}\right)=1.$ Then
		\begin{equation*}
			\begin{aligned}
				\left(\frac{-1,r}{\mathcal R}\right)
				&=
				\left(\frac{-1,r}{\mathcal R'}\right)
				=
				-1,\\[2mm]
				\left(\frac{\varepsilon_2,r}{\mathcal R}\right)
				&=
				1,\\[2mm]
				\left(\frac{\varepsilon_2,r}{\mathcal R'}\right)
				&=
				-1.
			\end{aligned}
		\end{equation*}
		
	Moreover, letting  $u\in \{\sqrt{\varepsilon_q},\sqrt{\varepsilon_{2q}}\}$, we get:
	 	\begin{enumerate}[$a)$]
			\item
			If $q\equiv 3 \pmod 8$, then there exist at least
			$u\in \{\sqrt{\varepsilon_q},\sqrt{\varepsilon_{2q}}\}$
			such that $\left(\frac{u,r}{\mathcal R}\right)
			=-1 \text{ or }
			\left(\frac{u,r}{\mathcal R'}\right)
			=
			-1,$ and $\overline{u}\notin \{\overline{1},\overline{\varepsilon_2}\}$ in the quotient $E_{F_1'}/(E_{F_1'}\cap N_{F_1/F_1'}(F_1)).$

			\item 
			If $q\equiv 7 \pmod 8$, then for all	$u\in \{\sqrt{\varepsilon_q},\sqrt{\varepsilon_{2q}}\},$ we have $\overline{u}\in \{\overline{1},-\overline{1},\overline{\varepsilon_2}\}$ in the quotient $ E_{F_1'}/(E_{F_1'}\cap N_{F_1/F_1'}(F_1)).$
		 	\end{enumerate}
	\end{enumerate}
 
\end{lemma}

\section{\bf The Proof of the  Main Theorem} \label{sec3}
 The proof of the main theorem is based on the following intermediate results.

\begin{lemma}\label{lemma1}
	Let $F:=\QQ(\sqrt{q},\sqrt{d})$, where $q\equiv 3 \pmod 4$ is a prime number and $d > 1$ is an odd square-free integer that is not divisible by $q$. Then $\rank(A(F)) = 3$ if and only if $F$ takes one of the following forms:
	\begin{enumerate}[\rm $1)$]
		\item  $F:=\QQ(\sqrt{q},\sqrt{d}),$ where $d=rs$ with $r$ and $s$ are two prime numbers such that 
		$\left( \frac{q}{r}\right) = \left( \frac{q}{s}\right) =1$ and $r\equiv s \equiv 1 \pmod 8$.

		\item $F:=\QQ(\sqrt{q},\sqrt{d}),$ where $d=rst$ with $r$, $s$ and $t$ are prime numbers satisfying $\left( \frac{q}{r}\right)=1$ and one of the following conditions:
		\begin{enumerate}[$a)$]

			\item $\left( \frac{q}{s}\right) =\left( \frac{q}{t}\right)=-1 $ and $r \equiv 1 \pmod 8$,
			
			\item $ \left( \frac{q}{s}\right) =-\left( \frac{q}{t}\right)=1 $ and $[(r\equiv s \equiv 3 \pmod 8)$, $(r\equiv s \equiv 5 \pmod 8)$, $(r\equiv s \equiv 7 \pmod 8)$ or $(r \equiv 1 \pmod 8 \text{ and } s \equiv 3 , 5 \text{ or } 7 \pmod 8 )]$,
			
			\item $ \left( \frac{q}{s}\right) =\left( \frac{q}{t}\right)=1 $ and $\exists\, q_1, q_2,q_3 \in \{ r, s, t \} $ such that $q_1 \equiv 3 \pmod 4, q_2 \equiv 3 $, $5  \pmod 8 $ and $\left( \frac{-1}{t}\right) \neq \left( \frac{2}{t}\right)$.
		\end{enumerate}
		\item  $F:=\QQ(\sqrt{q},\sqrt{d}),$ where $d=rstk$ with $r$, $s$, $t$ and $k$ are primes numbers satisfying  $\left( \frac{q}{k}\right) = \left( \frac{q}{t}\right) =-1 $  and one of the following conditions:
		\begin{enumerate}[$a)$]
			\item  $\left( \frac{q}{r}\right) = \left( \frac{q}{s}\right) =1 $ and $[(r \equiv 5 \pmod 8 \text{ and } s \equiv 3  \pmod 4 )$ or $(r \equiv 3\pmod 8 \text{ and } s \equiv  7  \pmod 8 )]$,

			\item $\left( \frac{q}{r}\right) =- \left( \frac{q}{s}\right) =1 $ and $r \not\equiv 1 \pmod 8$,
			
			\item $\left( \frac{q}{r}\right) = \left( \frac{q}{s}\right) =-1 $. 
			
		\end{enumerate}

	\end{enumerate}
\end{lemma}
 \begin{proof}
	Assume that  $F:=\QQ(\sqrt{q},\sqrt{d})$ where  $d > 1$ is an odd positive square-free integer that is not divisible by $q$. Let $t$ be the number of prime ideals of $K=\QQ(\sqrt{q})$ that are ramified in $F$.
Thus, $\rank(A(F)) = t_{F/K}-1-e_{F/K}$. As $e_{F/K} \in \{0,1,2 \} $, then the equality $\rank(A(F)) =3$ implies that $t_{F/K}=4, 5 \text{ or } 6$. The last equalities are equivalent to the fact that $d$ takes one of the following four forms:
$$rs \text{ with } \left( \frac{q}{r}\right) = \left( \frac{q}{s}\right) =1,\ rst \text{ with } \left( \frac{q}{r}\right) =1, \ rstk \text{ with } \left( \frac{q}{k}\right) = \left( \frac{q}{t}\right) =-1,$$

$$ rstke \text{ with } \left( \frac{q}{s}\right) = \left( \frac{q}{t}\right) =\left( \frac{q}{k}\right) = \left( \frac{q}{e}\right) =-1, $$
where $r, s, t, k$ and $e$ are prime numbers all different from q. In the following, we make use of \cite[Theorem 3.3]{Azmouh2-rank} in which the authors computed the index $e_{F/K}$.
\begin{enumerate}[\rm $1)$]
	
	\item Let $d=rs$, then
	\begin{enumerate}[\rm $\bullet$]
		
		\item  If $\left( \frac{q}{r}\right) = \left( \frac{q}{s}\right) =1$, then $t_{F/K}=4$ and so we have $\rank(A(F))= 3$ if and only if $ \left( \frac{-1}{r}\right) = \left( \frac{2}{r}\right)=\left( \frac{-1}{s}\right)=\left( \frac{2}{s}\right)=1$, which is equivalent to $r\equiv 1 \pmod 8$ and $s\equiv 1 \pmod 8$,
	\end{enumerate}
	
	\item Let $d=rst$ such that $\left( \frac{q}{r}\right) =1$, then we distinguish the following conditions:
	
	\begin{enumerate}[\rm $\bullet$]
		\item If $\left( \frac{q}{s}\right) = \left( \frac{q}{t}\right) =-1$, then $t_{F/K}=4$ and so  $\rank(A(F))= 3$ if and only if $r\equiv 1 \pmod 8$,
		
		\item  If $\left( \frac{q}{s}\right) = -\left( \frac{q}{t}\right) =1$, then $t_{F/K}=5$ and so $\rank(A(F))= 3$ if and only if $e_{F/K}=1$, which is  equivalent to $[r\equiv s \equiv 3 \pmod 8]$, $[r\equiv s \equiv 5 \pmod 8]$, $[r\equiv s \equiv 7 \pmod 8]$ or $[r \equiv 1 \pmod 8 \text{ and } s \equiv 3 , 5 \text{ or } 7 \pmod 8 ]$,
		
		\item  If $\left( \frac{q}{s}\right) = \left( \frac{q}{t}\right) =1$, then $t_{F/K}=6$ and so $\rank(A(F))= 3$ if and only if $\exists\, q_1, q_2,q_3 \in \{ r, s, t \} $ such that $q_1 \equiv 3 \pmod 4, q_2 \equiv 3 , 5  \pmod 8 $ and $\left( \frac{-1}{q_3}\right) \neq \left( \frac{2}{q_3}\right)$,
		
	\end{enumerate}
	
	\item  Let $d=rstk$ such that $\left( \frac{q}{k}\right) =-1$ and $\left( \frac{q}{t}\right) =-1$, we discuss the following conditions:
	\begin{enumerate}[\rm $\bullet$]
		
		\item If $\left( \frac{q}{r}\right) = \left( \frac{q}{s}\right) =1$, then $t_{F/K}=6$, hence $\rank(A(F))= 3$ if and only if $e_{F/K}=2$, which is equivalent to $[r \equiv 5 \pmod 8 \text{ and } s \equiv 3  \pmod 4 ]$ or $[r \equiv 3\pmod 8 \text{ and } s \equiv  7  \pmod 8 ]$,
		
		\item If $\left( \frac{q}{r}\right) = -\left( \frac{q}{s}\right) =1$, then $t_{F/K}=5$, hence $\rank(A(F))= 3$ if and only if  $e_{F/K}=1$, which is equivalent to $r \not\equiv 1 \pmod 8$,
		
		\item If $\left( \frac{q}{r}\right) = \left( \frac{q}{s}\right) =-1$, then $t_{F/K}=   4$, hence $\rank(A(F))= 3$ ,
	\end{enumerate}
	\item Let $d=rstke$ where $\left( \frac{q}{s}\right) = \left( \frac{q}{t}\right) =\left( \frac{q}{k}\right) = \left( \frac{q}{e}\right) =-1$, in this case $\rank(A(F))> 3$.
	\end{enumerate}	Combining all of the above, we obtain the desired result.
\end{proof}

\begin{lemma}\label{lemma3}
	
	Let $F:=\QQ(\sqrt{2q},\sqrt{d})$, where $q\equiv 3 \pmod 4$ is a prime number and $d \equiv 1 \pmod 4$ is a positive square-free integer.
	Let $\delta \in \{ 1,q \} $, then $\rank(A(F)) = 3$ if and only if $F$ takes one of the following forms:
	
	\begin{enumerate}[$1)$]
		\item 
		
		$d=\delta rs$, $\left( \frac{2q}{r} \right)= \left( \frac{2q}{s} \right) = 1$ and $[r\equiv s \equiv 1 \pmod 8]$,
		\item 
		$d=\delta rst$, $\left( \frac{2q}{r} \right)= 1$ and such that one of the following conditions is satisfied:
		
		\begin{enumerate}[$a)$]
			\item $\left( \frac{2q}{s} \right)=\left( \frac{2q}{t} \right)= -1$ and $[r\equiv 1 \pmod 8]$,
			
			\item $\left( \frac{2q}{s} \right)=-\left( \frac{2q}{t} \right)= 1$ and $[(r\equiv s \equiv 3 \pmod 8)$ or $(r\equiv s \equiv 5 \pmod 8)$ or $(r\equiv s \equiv 7 \pmod 8)$ or $(r \equiv 1 \pmod 8 \text{ and } s \equiv 3 , 5 \text{ or } 7 \pmod 8 )]$,
			
			\item $\left( \frac{2q}{s} \right)=\left( \frac{2q}{t} \right)= 1$ and $\exists\, q_1, q_2,q_3 \in \{ r, s, t \} $ such that  $\left( \frac{-1}{q_1} \right)=\left( \frac{2}{q_2} \right)=-1 \text{ and } \left( \frac{-1}{q_3} \right) \neq \left( \frac{2}{q_3} \right)$,
		\end{enumerate}
		
		\item $d=\delta rstk$, $\left( \frac{2q}{k} \right)= \left( \frac{2q}{t} \right)=-1$,  and such that one of the following conditions is satisfied:
		
		\begin{enumerate}[$a)$]
			\item $\left( \frac{2q}{r} \right)=\left( \frac{2q}{s} \right)= -1$,
			\item $\left( \frac{2q}{r} \right)=-\left( \frac{2q}{s} \right)= 1$ and $[r\equiv 3 \pmod 4$ or $r\equiv 5 \pmod 8]$,
			\item  $\left( \frac{2q}{r} \right)=\left( \frac{2q}{s} \right)= 1$ and  $[(r \equiv 3 \pmod 4 \text{ and } s \equiv  5  \pmod 8 )$ or $(r \equiv 7 \pmod 8 \text{ and } s \equiv  3  \pmod 8 )]$.
		\end{enumerate}

	\end{enumerate}

\end{lemma}

\begin{proof}
 	Assume that  $F:=\QQ(\sqrt{2q},\sqrt{d})$, where $q\equiv 3 \pmod 4$ is a prime number and $d \equiv 1 \pmod 4$ is a positive square-free integer. Let $\delta \in \{ 1,q \} $ and $t_{F/K}$ be the number of prime ideals of $K=\QQ(\sqrt{2q})$ that are ramified in $F$.
 	Thus, $\rank(A(F)) = t_{F/K}-1-e_{F/K}$. As $e_{F/K} \in \{0,1,2 \} $, then the equality $\rank(A(F)) =3$ implies that $t_{F/K}=4, 5 \text{ or } 6$. The last equalities are equivalent to the fact that $d$ takes one of the following four forms:
	 	\[
	\quad \delta rs \text{ with }\ \left( \frac{2q}{r} \right)= \left( \frac{2q}{s} \right) = 1 ,\quad \delta rst \ \text{with}\ 
	\left( \frac{2q}{r} \right) = 1,
	\]
	
	\[
	\delta rstk \ \text{with}\ 
	\left( \frac{2q}{k} \right)
	=
	\left( \frac{2q}{t} \right)
	= -1,
	\]
	
	\[
	\delta rstke \ \text{with}\ 
	\left( \frac{2q}{s} \right)
	=
	\left( \frac{2q}{t} \right)
	=
	\left( \frac{2q}{k} \right)
	=
	\left( \frac{2q}{e} \right)
	= -1,
	\]
	
	$
	\text{where } \delta \in \{1, q\} \text{ and } r, s, t, k \text{ and } e 
	\text{ are prime numbers all different from }$ $q$. In the following, we make use of \cite[Theorem 3.3]{Azmouh2-rank} in which the authors computed the index $e_{F/K}$.

		\begin{enumerate}[\rm $1.$]
		
		\item  Let $d=\delta rs$ and $\left( \frac{2q}{r} \right)= \left( \frac{2q}{s} \right) = 1$.
		As	$\rank(A(F)) = t_{F/K}-1-e_{F/K}$, and in this case we have $ t_{F/K}=4$, then $\rank(A(F)) = 3$ if and only if $e_{F/K}=0$, which is equivalent to $\left( \frac{2}{r} \right)= \left( \frac{-1}{r} \right) = \left( \frac{2}{s} \right)= \left( \frac{-1}{s} \right)=1$, which is equivalent to $r\equiv s \equiv 1 \pmod 8$,
		
		\item  Let $d=\delta rst$ and $\left( \frac{2q}{r} \right)= 1$, we discuss the following cases:
		\begin{enumerate}[\rm $\bullet$]
			\item If  $\left( \frac{2q}{s} \right)=\left( \frac{2q}{t} \right)= -1$, then  $\rank(A(F)) =3-e_{F/K}$ and it is equal to $3$ if and only if $e_{F/K}=0$, means $\left( \frac{2}{r} \right)= \left( \frac{-1}{r} \right) = 1$, which is equivalent to $r\equiv 1 \pmod 8$,
			
			\item  If  $\left( \frac{2q}{s} \right)=-\left( \frac{2q}{t} \right)= 1$, then  $\rank(A(F)) =4-e_{F/K}$ and it is equal to $3$ if and only if $e_{F/K}=1$, which is equivalent to $[r\equiv s \equiv 3 \pmod 8$, $r\equiv s \equiv 5 \pmod 8$, $r\equiv s \equiv 7 \pmod 8$ or $(r \equiv 1 \pmod 8 \text{ and } s \equiv 3 , 5 \text{ or } 7 \pmod 8 )]$, 
			
			\item  If  $\left( \frac{2q}{s} \right)=\left( \frac{2q}{t} \right)= 1$, then $\rank(A(F)) =5-e_{F/K}$ and it is equal to $3$ if and only if $e_{F/K}=2$, which is equivalent to the existence of  $ q_1, q_2,q_3 \in \{ r, s, t \} $ such that  $\left( \frac{-1}{q_1} \right)=\left( \frac{2}{q_2} \right)=-1 \text{ and } \left( \frac{-1}{q_3} \right) \neq \left( \frac{2}{q_3} \right)$,
		\end{enumerate}
		
		\item Let $d=\delta rstk$ and $\left( \frac{2q}{k} \right)= \left( \frac{2q}{t} \right)=-1$, we have the following conditions:
		\begin{enumerate}[\rm $\bullet$]
			\item If  $\left( \frac{2q}{r} \right)=\left( \frac{2q}{s} \right)= -1$, then $\rank(A(F)) =3-e_{F/K}$, under these conditions we have $\rank(A(F)) =3$,
			
			\item If $\left( \frac{2q}{r} \right)=-\left( \frac{2q}{s} \right)= 1$, then $\rank(A(F)) =4-e_{F/K}$ and it is equal to $3$ if and only if $e_{F/K}=1$, means $\left( \frac{2}{r} \right)=-1 $ or $ \left( \frac{-1}{r} \right) =-1$,
			
			\item If  $\left( \frac{2q}{r} \right)=\left( \frac{2q}{s} \right)= 1$, then $\rank(A(F)) =5-e_{F/K}$  and it is equal to $3$ if and only if $e_{F/K}=2$, means one of the two following conditions is satisfied: $(r \equiv 3 \pmod 4 \text{ and } s \equiv  5  \pmod 8 )$ or $(r \equiv 7 \pmod 8 \text{ and } s \equiv  3  \pmod 8 )$,
		\end{enumerate}
		
		\item Let $d=\delta rstke$ and $\left( \frac{2q}{s} \right)
		=
		\left( \frac{2q}{t} \right)
		=
		\left( \frac{2q}{k} \right)
		=
		\left( \frac{2q}{e} \right)
		= -1,$
		in this case $\rank(A(F))> 3$.
	\end{enumerate}
 	This completes the proof.
\end{proof}

Let us recall that $L$ is named to be any real biquadratic field of the form
$L:=\mathbb Q(\sqrt{q},\sqrt r),$ where $r\equiv 1 \pmod 8$ and $\left(\frac{q}{r}\right)=1$.
We have:

\begin{lemma}\label{lemma4}
 	Let $F:=\QQ(\sqrt{q},\sqrt{d}) \neq L$ and $F_{1}:=\QQ(\sqrt{q},\sqrt{d},\sqrt{2})$, 
	where $q\equiv 3 \pmod 4$ is a prime number and $d>2$ is a square-free integer relatively prime to $q$.
	Assume that $d=rst$, where $r$, $s$ and $t$ are prime numbers such that $r$ and $s$ are not congruent to $1$ mod $8$, and such that $\rank(A(F))=3 $ $($cf. Lemma \ref{lemma1}-$2)$$)$.
	Then,  $\rank(A(F_{1}))=3 $ if and only if one of the  following  conditions holds:
	\begin{enumerate}[$1)$]
		\item  $ \left( \frac{q}{s}\right)=\left( \frac{q}{r}\right)=-\left( \frac{q}{t}\right)=1 $ and we have one of the following congruence conditions: 
		
		\begin{enumerate}[\rm $a)$]
			\item  	$r\equiv s \equiv 3 \pmod 8$  and $[t \equiv 5 \pmod 8, \ (t \equiv 3 \pmod 8 \text{ and }q \equiv 7 \pmod 8) \text{ or } (t \equiv 7 \pmod 8 \text{ and }q \equiv 3 \pmod 8)]$,
			\item $r\equiv s \equiv 5 \pmod 8$  and $ [t \equiv 3 \pmod 8 \text{ or } t \equiv 5 \pmod 8 \text{ or }  (t \equiv 7 \pmod 8 \text{ and } q \equiv 3 \pmod 8 )]$,
		 \end{enumerate}
		
		\item $ \left( \frac{q}{s}\right)=\left( \frac{q}{r}\right)=\left( \frac{q}{t}\right)=1 $ and such that one of the following conditions is satisfied:
		\begin{enumerate}[\rm $a)$]
			\item $q \equiv 3 \pmod 8$, $r\equiv t\equiv 3 \pmod 8$ and $s \equiv 5 \pmod 8$,
			\item  $q \equiv 3 \pmod 8$, $r \equiv 3\pmod 8$ and $s \equiv t \equiv 5 \pmod 8$,

			\item  $q \equiv 3 \pmod 8$, $r \equiv 3\pmod 8$, $s \equiv 5 \pmod 8$, $ t \equiv 7 \pmod 8$ and $\overline{\sqrt{\varepsilon_{2q}}} \notin \langle \overline{-1}, \overline{\varepsilon_2}, \overline{\sqrt{\varepsilon_{q}}} \rangle$, 
			
			\item  $q \equiv 3 \pmod 8$, $r \equiv t \equiv 5\pmod 8$, $s \equiv  7 \pmod 8$ and  $\overline{\sqrt{\varepsilon_{2q}}},  \notin \langle \overline{-1}, \overline{\varepsilon_2}, \overline{\sqrt{\varepsilon_{q}}} \rangle$,
			
			\item $q \equiv 7 \pmod 8$, $r \equiv t \equiv 3 \pmod 8$  and $s  \equiv 5 \pmod 8$,
			
			\item $q \equiv 7 \pmod 8$, $r  \equiv 3 \pmod 8$  and $s \equiv t  \equiv 5 \pmod 8$,
			
			\item  $q \equiv 3 \pmod 8$, $r\equiv 3 \pmod 8$, $s\equiv 7 \pmod 8$, $t \equiv 3 \pmod 8$ and  $\overline{\sqrt{\varepsilon_{2q}}},\overline{\sqrt{\varepsilon_{q}}}  \notin \langle \overline{-1}, \overline{\varepsilon_2} \rangle$.
		\end{enumerate}
 	\end{enumerate}
	
\end{lemma}
\begin{proof}
	Let $F'_{1}:=\QQ(\sqrt{q},\sqrt{2})$. Throughout this proof, we use the norm computations established in Lemma \ref{chemseddin}. Assume that $r$, $s$ and $t$ satisfy the conditions of  Lemma \ref{lemma1}-$2)$. We distinguish the following cases:
		\begin{enumerate}[$1)$] 
		\item Assume that $ \left( \frac{q}{s}\right)=\left( \frac{q}{r}\right)=-\left( \frac{q}{t}\right)=1 $ $($cf. Lemma \ref{lemma1}-$2-b)$$)$. we discuss the following subcases:
		
				\begin{enumerate}[$A)$] 
		\item  If $r\equiv s \equiv 3 \pmod 8$, then 
		$\rank(A(F_{1}))=t_{F_{1}/F'_{1}}-1-e_{F_{1}/F'_{1}}=5-e_{F_{1}/F'_{1}}$. We have
		
		\[
		\begin{aligned}
			\left( \frac{-1,d}{\mathcal{R}} \right) &=\left( \frac{-1,d}{\mathcal{S}} \right)= 1\\
			\left( \frac{\varepsilon_2,d}{\mathcal{R}}  \right) &=\left( \frac{\varepsilon_2,d}{\mathcal{S}}  \right)= -1\\
			\left( \frac{\sqrt{\varepsilon_q},d}{\mathcal{R}} \right) &=\left( \frac{\sqrt{\varepsilon_q},d}{\mathcal{S}} \right)= 1\\
			\left( \frac{\sqrt{\varepsilon_{2q}},d}{\mathcal{R}} \right) &=\left( \frac{\sqrt{\varepsilon_{2q}},d}{\mathcal{S}} \right)= \left( \frac{(-1)^{\gamma+1}}{r} \right).
		\end{aligned}
		\]
		Moreover, we have the following cases for $t$:
		\[
		\begin{array}{l|l}
			\begin{aligned}
				t \equiv 3 \text{ or } 5 \pmod{8}\\
				\left( \frac{-1,d}{\mathcal{T}} \right) &= 1\\
				\left( \frac{\varepsilon_2,d}{\mathcal{T}} \right) &= \left( \frac{-1}{t} \right)\\
				\left( \frac{\sqrt{\varepsilon_q},d}{\mathcal{T}} \right) &= \left( \frac{(-1)^{\gamma+1}}{t} \right)\\
				\left( \frac{\sqrt{\varepsilon_{2q}},d}{\mathcal{T}} \right) &= \left( \frac{-1}{t} \right)\left( \frac{2}{t} \right)
			\end{aligned}
			&
			\begin{aligned}
				t \equiv 1 \text{ or } 7 \pmod{8}\\
				\left( \frac{-1,d}{\mathcal{T}} \right) &= 1\\
				\left( \frac{\varepsilon_2,d}{\mathcal{T}} \right) &= 1\\
				\left( \frac{\sqrt{\varepsilon_q},d}{\mathcal{T}} \right) &= \left( \frac{(-1)^{\gamma}}{t} \right)\\
				\left( \frac{\sqrt{\varepsilon_{2q}},d}{\mathcal{T}} \right) &= \left( \frac{(-1)^{\gamma}}{t} \right)
			\end{aligned}
		\end{array}
		\]
		
		where $\mathcal{R}$ (resp. $\mathcal{S}$, ${\mathcal{T}}$) is a prime ideal of $F'_{1}$ above $r$ (resp. $s$, $t$). Now, we examine each case of $t$ individually:
			\begin{enumerate}[\rm $\bullet$]
		\item Let $t \equiv 5 \pmod 8$. Using the above norm residue symbols we check that $$
		{E_{F'_1}}/({E_{F'_1}\cap N_{F_1/F'_1}(F_1)})=\{ \overline{1}, \overline{\varepsilon_2}, \overline{\sqrt{\varepsilon_{2q}}}, \overline{\varepsilon_2\sqrt{\varepsilon_{2q}}}\}.$$
		Thus, $e_{F_{1}/F'_{1}}=2$ and so $\rank(A(F_{1})=3$.
		
		\item  Let $t \equiv 3 \pmod 8$, then using the above norm residue symbols we check that 
		$
		\left( \frac{\sqrt{\varepsilon_q},d}{\mathcal{T}} \right) = \left( \frac{(-1)^{\gamma+1}}{t} \right)=-1 \text{ or } 1 
		$ according to $q \equiv 7\pmod 8$ or not. Likewise, 
		$\left( \frac{\sqrt{\varepsilon_{2q}},d}{\mathcal{R}} \right) = \left( \frac{(-1)^{\gamma+1}}{r} \right)=-1 \text{ or } 1 
		$ according to $q \equiv 7\pmod 8$ or not, and in this case $\rank(A(F_{1}) \in \{ 3,4\}.$
		
		More precisely, suppose we have $q \equiv 3 \pmod 8$, then 
		$ {E_{F'_1}}/({E_{F'_1}\cap N_{F_1/F'_1}(F_1)})=\{ \overline{1}, \overline{\varepsilon_2} \}$, thus, $e_{F_{1}/F'_{1}}=1$, hence $\rank(A(F_{1})=4$. Now suppose $q \equiv 7 \pmod 8$, then we have 
		$ {E_{F'_1}}/({E_{F'_1}\cap N_{F_1/F'_1}(F_1)})=\{ \overline{1}, \overline{\epsilon_2}, \overline{\sqrt{\varepsilon_q}}, \overline{\varepsilon_{2}\sqrt{\varepsilon_q} \}}$, thus, $e_{F_{1}/F'_{1}}=2$. Hence $\rank(A(F_{1})=3$.
		
		\item  Let $t \equiv 7\pmod 8$, in this case we have $\rank(A(F_{1}) \in \{ 3,4\}.$  More precisely, if $q \equiv 3 \pmod 8$, we have 
		$ {E_{F'_1}}/({E_{F'_1}\cap N_{F_1/F'_1}(F_1)})=\{ \overline{1}, \overline{\epsilon_2}, \overline{\sqrt{\varepsilon_q}}, \overline{\varepsilon_{2}\sqrt{\varepsilon_q} \}}$, thus $e_{F_{1}/F'_{1}}=2$. Hence $\rank(A(F_{1})=3$. Now  if $q \equiv 7 \pmod 8$, then  $
		\left( \frac{\sqrt{\varepsilon_q},d}{\mathcal{T}} \right)=\left( \frac{\sqrt{\varepsilon_{2q}},d}{\mathcal{T}} \right) = \left( \frac{(-1)^{\gamma}}{t} \right)=1
		$, $
		\left( \frac{\sqrt{\varepsilon_{2q}},d}{\mathcal{R}} \right) = \left( \frac{(-1)^{\gamma+1}}{r} \right)=-1  
		$ and $\left( \frac{\varepsilon_{2}\sqrt{\varepsilon_{2q}},d}{\mathcal{R}} \right) =\left( \frac{\varepsilon_{2},d}{\mathcal{R}} \right) \left( \frac{\sqrt{\varepsilon_{2q}},d}{\mathcal{R}} \right)  =1
		$, hence $ {E_{F'_1}}/({E_{F'_1}\cap N_{F_1/F'_1}(F_1)})=\{ \overline{1}, \overline{\epsilon_2} \}$ and so $e_{F_{1}/F'_{1}}=1$. Hence $\rank(A(F_{1})=4$.
		
		\item  Let $t \equiv 1\pmod 8$, in this case we have $
		\left( \frac{\sqrt{\varepsilon_q},d}{\mathcal{T}} \right)=\left( \frac{\sqrt{\varepsilon_{2q}},d}{\mathcal{T}} \right) = \left( \frac{(-1)^{\gamma}}{t} \right)=1
		$, and $\left( \frac{\sqrt{\varepsilon_{2q}},d}{\mathcal{R}} \right)=\left( \frac{(-1)^{\gamma+1}}{r} \right)= 1 \text{ or } -1$ according to $q \equiv 3 \pmod 8$ or not, and we check that $\left( \frac{\varepsilon_{2}\sqrt{\varepsilon_{2q}},d}{\mathcal{R}} \right) =\left( \frac{\varepsilon_{2},d}{\mathcal{R}} \right) \left( \frac{\sqrt{\varepsilon_{2q}},d}{\mathcal{R}} \right)  =1
		$. Hence $ {E_{F'_1}}/({E_{F'_1}\cap N_{F_1/F'_1}(F_1)})=\{ \overline{1}, \overline{\epsilon_2} \}$ and so $e_{F_{1}/F'_{1}}=1$, and as a result we have $\rank(A(F_{1})=4$.
			\end{enumerate}
		\item  If  $r\equiv s \equiv 5 \pmod 8$, it follows, by  Lemma \ref{chemseddin}, that  we have
	 	\[
		\begin{aligned}
			\left( \frac{-1,d}{\mathcal{R}} \right) &=\left( \frac{-1,d}{\mathcal{S}} \right)= 1,\\
			\left( \frac{\varepsilon_2,d}{\mathcal{R}}  \right) &=\left( \frac{\varepsilon_2,d}{\mathcal{S}}  \right)= 1,\\
			\left( \frac{\sqrt{\varepsilon_q},d}{\mathcal{R}} \right) &=\left( \frac{\sqrt{\varepsilon_q},d}{\mathcal{S}} \right)= -1,\\
			\left( \frac{\sqrt{\varepsilon_{2q}},d}{\mathcal{R}} \right) &=\left( \frac{\sqrt{\varepsilon_{2q}},d}{\mathcal{S}} \right)= 1.
		\end{aligned}
		\]
	More precisely, we distinguish the following sub-cases:
	
		\begin{enumerate}[\rm $\bullet$]
		\item  Let $t \equiv 3\pmod 8$, then $\left( \frac{\varepsilon_2,d}{\mathcal{T}} \right)=\left( \frac{\sqrt{\varepsilon_q},d}{\mathcal{R}} \right)=-1$  and $\left( \frac{\sqrt{\varepsilon_{2q}},d}{\mathcal{T}} \right) =1
		$. Therefore, we have $$ {E_{F'_1}}/({E_{F'_1}\cap N_{F_1/F'_1}(F_1)})=\{ \overline{1}, \overline{\epsilon_2}, \overline{\varepsilon_q},\overline{\varepsilon_2\sqrt{\varepsilon_q}} \},$$
		 and so $e_{F_{1}/F'_{1}}=2$, and $\rank(A(F_{1})=3$.
		
		\item  Let $t \equiv 5\pmod 8$, then  $\left( \frac{\varepsilon_2,d}{\mathcal{T}} \right)=1$, $\left( \frac{\sqrt{\varepsilon_q},d}{\mathcal{R}} \right)=-1$ and $\left( \frac{\sqrt{\varepsilon_{2q}},d}{\mathcal{T}} \right) =-1
		$. Thus, $ {E_{F'_1}}/({E_{F'_1}\cap N_{F_1/F'_1}(F_1)})=\{ \overline{1}, \overline{\sqrt{\varepsilon_q}}, \overline{\sqrt{\varepsilon_{2q}}},\overline{\sqrt{\varepsilon_q}\sqrt{\varepsilon_{2q}}} \}$. Therefore, $e_{F_{1}/F'_{1}}=2$  and $\rank(A(F_{1})=3$.
	\end{enumerate}
		\item  If $r\equiv s \equiv 7 \pmod 8$, we deduce that $\rank(A(F_{1}))=t_{F_{1}/F'_{1}}-1-e_{F_{1}/F'_{1}}=9-e_{F_{1}/F'_{1}} >3$.
	\end{enumerate}

		\item  Assume that $ \left( \frac{q}{s}\right)=\left( \frac{q}{r}\right)=\left( \frac{q}{t}\right)=1$ $($cf. Lemma \ref{lemma1}-$2-c)$$)$.
	 	We distinguish two cases according to the congruence of $q$ modulo $8$:
		
		\ding{229} Suppose that $q\equiv 3 \pmod 8$. We have the following sub-cases according to the congruence classes of
		$r$, $s$, and $t$ modulo $8$:
		\begin{enumerate}[\rm $\bullet$]
		\item Let $r\equiv t \equiv 3 \pmod 8$ and $s\equiv 5 \pmod 8$. We have, $\rank(A(F_{1}))=t_{F_{1}/F'_{1}}-1-e_{F_{1}/F'_{1}}=6-1-e_{F_{1}/F'_{1}}=5-e_{F_{1}/F'_{1}}$, and it is equal to $3$ if and only if $e_{F_{1}/F'_{1}}=2$. Moreover, we have the following  Hilbert symbols:
		\[
		\begin{aligned}
			\left(\frac{-1,d}{\mathcal{R}}\right)
			&= \left(\frac{-1,d}{\mathcal{S}}\right)
			= \left(\frac{-1,d}{\mathcal{T}}\right)
			= 1 \\
			\left(\frac{\varepsilon_2,d}{\mathcal{R}}\right)
			&= -1 \\
			\left(\frac{\sqrt{\varepsilon_q},d}{\mathcal{R}}\right)
			&= 1, 
			\left(\frac{\sqrt{\varepsilon_q},d}{\mathcal{S}}\right)
			= -1 \\
			\left(\frac{\sqrt{\varepsilon_{2q}},d}{\mathcal{R}}\right)
			&= \left(\frac{\sqrt{\varepsilon_{2q}},d}{\mathcal{S}}\right)
			= \left(\frac{\sqrt{\varepsilon_{2q}},d}{\mathcal{T}}\right)
			= 1.
		\end{aligned}
		\]
	
		Hence $ {E_{F'_1}}/({E_{F'_1}\cap N_{F_1/F'_1}(F_1)})=\{ \overline{1}, \overline{\varepsilon_2}, \overline{\sqrt{\varepsilon_{q}}},\overline{\varepsilon_2\sqrt{\varepsilon_{q}}} \}$ and so $e_{F_{1}/F'_{1}}=2$. Thus $\rank(A(F_{1})=3$.
		
		\medskip

			\item Let $r\equiv 3 \pmod 8$ and $s\equiv t \equiv 5 \pmod 8$. As in the previous sub-case, we get:
		\[
		\begin{aligned}
			\left(\frac{-1,d}{\mathcal{R}}\right)
			&= \left(\frac{-1,d}{\mathcal{S}}\right)
			= \left(\frac{-1,d}{\mathcal{T}}\right)
			= 1,\\
			\left(\frac{\sqrt{\varepsilon_{2q}},d}{\mathcal{R}}\right)
			&= \left(\frac{\sqrt{\varepsilon_{2q}},d}{\mathcal{S}}\right)
			= \left(\frac{\sqrt{\varepsilon_{2q}},d}{\mathcal{T}}\right)
			= 1,\\
			\left(\frac{\varepsilon_2,d}{\mathcal{R}}\right)
			&= -1,\\
			\left(\frac{\sqrt{\varepsilon_q},d}{\mathcal{T}}\right)
			&= -1, 
			\left(\frac{\sqrt{\varepsilon_q},d}{\mathcal{R}}\right)
			= 1.
		\end{aligned}
		\]
       
		Hence $ {E_{F'_1}}/({E_{F'_1}\cap N_{F_1/F'_1}(F_1)})=\{ \overline{1}, \overline{\varepsilon_2}, \overline{\sqrt{\varepsilon_{q}}},\overline{\varepsilon_2\sqrt{\varepsilon_{q}}} \}$ and so $e_{F_{1}/F'_{1}}=2$. Hence $\rank(A(F_{1})=3$.
		
		\medskip

		\item  Now let $r\equiv 3 \pmod 8$, $s\equiv 5 \pmod 8$ and $t \equiv 7 \pmod 8$. Then
		$\rank(A(F_{1}))=t_{F_{1}/F'_{1}}-1-e_{F_{1}/F'_{1}}=8-1-e_{F_{1}/F'_{1}}= 7-e_{F_{1}/F'_{1}}$. Notice that , by Lemma \ref{chemseddin} , we have:

		$$	\left( \frac{-1,d}{\mathcal{R}} \right) = 1, ~~
		\left( \frac{\varepsilon_2,d}{\mathcal{R}} \right) =  -1, ~~
		\left( \frac{\sqrt{\varepsilon_{2q}},d}{\mathcal{R}} \right) =1, ~~
		\left( \frac{\sqrt{\varepsilon_{q}},d}{\mathcal{R}} \right) =1$$
		
 and

		$$	\left( \frac{-1,d}{\mathcal{S}} \right) = 1, ~~
		\left( \frac{\varepsilon_2,d}{\mathcal{S}} \right) =1, ~~
		\left( \frac{\sqrt{\varepsilon_{2q}},d}{\mathcal{S}} \right) =1, ~~
		\left( \frac{\sqrt{\varepsilon_{q}},d}{\mathcal{S}} \right) =  -1.$$
		
	 Moreover, we have
	 	$	\left( \frac{-1,d}{\mathcal{T}} \right) = \left( \frac{-1,d}{\mathcal{T'}} \right)=-1$ and
		$\left( \frac{\varepsilon_2,d}{\mathcal{T}} \right) =- \left( \frac{\varepsilon_2,d}{\mathcal{T'}} \right)=1$.
		We notice that $-1$, $\varepsilon_2 $ and $\sqrt{\varepsilon_q}$ are not norms and linearly independent  in the quotient $ {E_{F'_1}}/({E_{F'_1}\cap N_{F_1/F'_1}(F_1)})$, and based on Lemma \ref{chemseddin}-$3)$, there exist at least $u \in \{ \sqrt{\varepsilon_q}, \sqrt{\varepsilon_{2q}}\}$ such that [$\left( \frac{u,d}{\mathcal{T}} \right)=-1$ or $\left( \frac{u,d}{\mathcal{T'}} \right)=-1$]  and $\overline{u} \notin \{ \overline{-1}, \overline{\varepsilon_{2}}\} $ in the quotient $ {E_{F'_1}}/({E_{F'_1}\cap N_{F_1/F'_1}(F_1)})$.
		
	It follows that if  $\sqrt{\varepsilon_{2q}}$ is not a norm, then $e_{F_{1}/F'_{1}}=4$ and so $\rank(A(F_{1})=3$. Otherwise, $e_{F_{1}/F'_{1}}=3$ and $\rank(A(F_{1})=4$.
		
		\medskip
		
		\item  If $r\equiv 3 \pmod 8$, $s\equiv 7 \pmod 8$, $t \equiv 7 \pmod 8$, then
		$\rank(A(F_{1}))=t_{F_{1}/F'_{1}}-1-e_{F_{1}/F'_{1}}=10-1-e_{F_{1}/F'_{1}}= 9-e_{F_{1}/F'_{1}}$. We compute the Hilbert symbols:

		$$	\left( \frac{-1,d}{\mathcal{R}} \right) = 1, ~~
		\left( \frac{\varepsilon_2,d}{\mathcal{R}} \right)  =-1, ~~
		\left( \frac{\sqrt{\varepsilon_{2q}},d}{\mathcal{R}} \right) =1, ~~
		\left( \frac{\sqrt{\varepsilon_{q}},d}{\mathcal{R}} \right) =1$$,

		$$	\left( \frac{-1,d}{\mathcal{S}} \right) = \left( \frac{-1,d}{\mathcal{S'}} \right)=-1, ~~
		\left( \frac{\varepsilon_2,d}{\mathcal{S}} \right) =- \left( \frac{\varepsilon_2,d}{\mathcal{S'}} \right)=1$$
		
	and
	
		$$\left( \frac{-1,d}{\mathcal{T}} \right) = \left( \frac{-1,d}{\mathcal{T'}} \right)=-1, ~~
		\left( \frac{\varepsilon_2,d}{\mathcal{T}} \right) =- \left( \frac{\varepsilon_2,d}{\mathcal{T'}} \right)=1.$$
		
       In this case $\rank(A(F_{1})) \in \{5,6\}$
		
		\medskip
		
		\item If $r\equiv 3 \pmod 8$, $s\equiv 7 \pmod 8$, $t \equiv 3 \pmod 8$, then
		$\rank(A(F_{1}))=t_{F_{1}/F'_{1}}-1-e_{F_{1}/F'_{1}}=8-1-e_{F_{1}/F'_{1}}= 7-e_{F_{1}/F'_{1}}$, and we have the following  Hilbert symbols:

		$$	\left( \frac{-1,d}{\mathcal{R}} \right) = 1, ~~
		\left( \frac{\varepsilon_2,d}{\mathcal{R}} \right) =-1, ~~
		\left( \frac{\sqrt{\varepsilon_{2q}},d}{\mathcal{R}} \right) =1, ~~
		\left( \frac{\sqrt{\varepsilon_{q}},d}{\mathcal{R}} \right) =1,$$

		$$	\left( \frac{-1,d}{\mathcal{S}} \right) = \left( \frac{-1,d}{\mathcal{S'}} \right)=-1, ~~	\left( \frac{\varepsilon_2,d}{\mathcal{S}} \right) =- \left( \frac{\varepsilon_2,d}{\mathcal{S'}} \right)=1$$
		and	
		$$	\left( \frac{-1,d}{\mathcal{T}} \right) = 1, ~~
		\left( \frac{\varepsilon_2,d}{\mathcal{T}} \right) =-1, ~~
		\left( \frac{\sqrt{\varepsilon_{2q}},d}{\mathcal{T}} \right)=1, ~~
		\left( \frac{\sqrt{\varepsilon_{q}},d}{\mathcal{T}} \right) =1.$$

		In this case $\rank(A(F_{1})) \in \{3,4\}$.  We check that it is  equal to $3$ if and only if $e_{F_{1}/F'_{1}}=4$. This case arises when $q\equiv 3 \pmod 8$ and both $\sqrt{\varepsilon_{q}}$ and $\sqrt{\varepsilon_{2q}}$ are not norms.
		
		\medskip
		
		\item If $r\equiv 5 \pmod 8$, $s\equiv 7 \pmod 8$, $t \equiv 5 \pmod 8$, then 
		$\rank(A(F_{1}))=t_{F_{1}/F'_{1}}-1-e_{F_{1}/F'_{1}}=8-1-e_{F_{1}/F'_{1}}= 7-e_{F_{1}/F'_{1}}$. We have

		$$\left( \frac{-1,d}{\mathcal{R}} \right) = 1, ~~
		\left( \frac{\varepsilon_2,d}{\mathcal{R}} \right)=1, ~~
		\left( \frac{\sqrt{\varepsilon_{2q}},d}{\mathcal{R}} \right)=1, ~~
		\left( \frac{\sqrt{\varepsilon_{q}},d}{\mathcal{R}} \right) )=-1,$$

		$$	\left( \frac{-1,d}{\mathcal{S}} \right) = \left( \frac{-1,d}{\mathcal{S'}} \right)=-1, ~~
		\left( \frac{\varepsilon_2,d}{\mathcal{S}} \right) =- \left( \frac{\varepsilon_2,d}{\mathcal{S'}} \right)=1$$
	 and
		$$	\left( \frac{-1,d}{\mathcal{T}} \right) = 1, ~~
		\left( \frac{\varepsilon_2,d}{\mathcal{T}} \right)=1, ~~
		\left( \frac{\sqrt{\varepsilon_{2q}},d}{\mathcal{T}} \right)=1, ~~
		\left( \frac{\sqrt{\varepsilon_{q}},d}{\mathcal{T}} \right) =-1.$$
		
		In this case $\rank(A(F_{1})) \in \{3,4\}$.
		
		\medskip

		\item If $r\equiv 5 \pmod 8$, $s\equiv 7 \pmod 8$, $t \equiv 7 \pmod 8$, then
		$\rank(A(F_{1}))=t_{F_{1}/F'_{1}}-1-e_{F{1}/F'_{1}}=10-1-e_{F{1}/F'_{1}}= 9-e_{F{1}/F'_{1}}$. Moreover, we have

		$$	\left( \frac{-1,d}{\mathcal{R}} \right) = 1, ~~
		\left( \frac{\varepsilon_2,d}{\mathcal{R}} \right)=1, ~~
		\left( \frac{\sqrt{\varepsilon_{2q}},d}{\mathcal{R}} \right)=1, ~~
		\left( \frac{\sqrt{\varepsilon_{q}},d}{\mathcal{R}} \right) =-1,$$

		$$	\left( \frac{-1,d}{\mathcal{S}} \right) = \left( \frac{-1,d}{\mathcal{S'}} \right)=-1, ~~
		\left( \frac{\varepsilon_2,d}{\mathcal{S}} \right) =- \left( \frac{\varepsilon_2,d}{\mathcal{S'}} \right)=1$$ 
	and		
	 	$$	\left( \frac{-1,d}{\mathcal{T}} \right) = \left( \frac{-1,d}{\mathcal{T'}} \right)=-1, ~~
		\left( \frac{\varepsilon_2,d}{\mathcal{T}} \right) =- \left( \frac{\varepsilon_2,d}{\mathcal{T'}} \right)=1,$$

		In this case $\rank(A(F_{1})) \in \{5,6\}$.
	\end{enumerate}
		\ding{229}	Suppose $q\equiv 7 \pmod 8$. We have: 
		\begin{enumerate}[\rm $\bullet$]
		\item If $r\equiv t \equiv 3 \pmod 8$, $s\equiv 5 \pmod 8$,
		then  $\rank(A(F_{1}))=t_{F_{1}/F'_{1}}-1-e_{F_{1}/F'_{1}}=6-1-e_{F_{1}/F'_{1}}=5-e_{F_{1}/F'_{1}}$, and it is  equal to $3$ if and only if $e_{F_{1}/F'_{1}}=2$. We have:
		
		$$\left( \frac{-1,d}{\mathcal{R}} \right)=-\left( \frac{\sqrt{\varepsilon_{2q}},d}{\mathcal{R}} \right)=-\left( \frac{\varepsilon_2,d}{\mathcal{R}} \right)=\left( \frac{\sqrt{\varepsilon_q},d}{\mathcal{R}} \right)=1,$$

		$$\left( \frac{-1,d}{\mathcal{S}} \right)=\left( \frac{\sqrt{\varepsilon_{2q}},d}{\mathcal{S}} \right)=-\left( \frac{\sqrt{\varepsilon_q},d}{\mathcal{S}} \right)=1,$$ 
		and	
		$$\left( \frac{-1,d}{\mathcal{T}} \right)=-\left( \frac{\sqrt{\varepsilon_{2q}},d}{\mathcal{T}} \right)=1.
		$$
	Hence $ {E_{F'_1}}/({E_{F'_1}\cap N_{F_1/F'_1}(F_1)})=\{ \overline{1}, \overline{\varepsilon_2}, \overline{\sqrt{\varepsilon_{q}}},\overline{\varepsilon_2\sqrt{\varepsilon_{q}}} \}$ and so $e_{F_{1}/F'_{1}}=2$. Thus $\rank(A(F_{1})=3$.
		
		\item  If $r\equiv 3 \pmod 8$ and $s\equiv t \equiv 5 \pmod 8$, then 
		$$\left( \frac{-1,d}{\mathcal{R}} \right)=\left( \frac{\sqrt{\varepsilon_{2q}},d}{\mathcal{R}} \right)=\left( \frac{\sqrt{\varepsilon_q},d}{\mathcal{R}} \right)=-\left( \frac{\varepsilon_2,d}{\mathcal{R}} \right)=1,$$
		
		$$ \left( \frac{-1,d}{\mathcal{S}} \right)=\left( \frac{\sqrt{\varepsilon_{2q}},d}{\mathcal{S}} \right)=1,
		$$
		
		$$\left( \frac{\sqrt{\varepsilon_{2q}},d}{\mathcal{T}} \right)=\left( \frac{-1,d}{\mathcal{T}} \right)=-\left( \frac{\sqrt{\varepsilon_q},d}{\mathcal{T}} \right)=1.$$ 
		Hence $ {E_{F'_1}}/({E_{F'_1}\cap N_{F_1/F'_1}(F_1)})=\{ \overline{1}, \overline{\varepsilon_2}, \overline{\sqrt{\varepsilon_{2q}}},\overline{\varepsilon_2\sqrt{\varepsilon_{q}}} \}$ and so $e_{F_{1}/F'_{1}}=2$. As a result $\rank(A(F_{1})=3$
		
		\item  If $r\equiv 3 \pmod 8$, $s\equiv 5 \pmod 8$ and $t \equiv 7 \pmod 8$, then $ {E_{F'_1}}/({E_{F'_1}\cap N_{F_1/F'_1}(F_1)})=\langle \overline{-1}, \overline{\varepsilon_2}, \overline{\sqrt{\varepsilon_{q}}} \rangle$ and so $e_{F_{1}/F'_{1}}=3$, and $\rank(A(F_{1})=4$.
		
		\item  If $r\equiv 3 \pmod 8$, $s\equiv 7 \pmod 8$ and $t \equiv 7 \pmod 8$, then
		$\rank(A(F_{1}))=t_{F_{1}/F'_{1}}-1-e_{F_{1}/F'_{1}}=10-1-e_{F_{1}/F'_{1}}= 9-e_{F_{1}/F'_{1}}>3$.
		
		\item  If $r\equiv 3 \pmod 8$, $s\equiv 7 \pmod 8$ and $t \equiv 3 \pmod 8$, then 	$\rank(A(F_{1}))=t_{F_{1}/F'_{1}}-1-e_{F_{1}/F'_{1}}=8-1-e_{F_{1}/F'_{1}}= 7-e_{F_{1}/F'_{1}}>3$

		\item  If $r\equiv 5 \pmod 8$, $s\equiv 7 \pmod 8$ and $t \equiv 5 \pmod 8$, then $ {E_{F'_1}}/({E_{F'_1}\cap N_{F_1/F'_1}(F_1)})=\langle \overline{-1}, \overline{\varepsilon_2}, \overline{\sqrt{\varepsilon_{q}}} \rangle$ and so $e_{F_{1}/F'_{1}}=3$. Thus 
		$\rank(A(F_{1}))=t_{F_{1}/F'_{1}}-1-e_{F_{1}/F'_{1}}=8-1-e_{F_{1}/F'_{1}}= 7-e_{F_{1}/F'_{1}}=4$.

		\item  If $r\equiv 5 \pmod 8$, $s\equiv 7 \pmod 8$ and $t \equiv 7 \pmod 8$, then $\rank(A(F_{1}))=t_{F_{1}/F'_{1}}-1-e_{F_{1}/F'_{1}}=10-1-e_{F_{1}/F'_{1}}= 9-e_{F_{1}/F'_{1}}>3$.
		
		\end{enumerate}
	\end{enumerate}
	
\end{proof}

\begin{lemma}\label{lemma41}
	
	Let $F:=\QQ(\sqrt{q},\sqrt{d})$ and $F_{1}:=\QQ(\sqrt{q},\sqrt{d},\sqrt{2})$ 
	where $q\equiv 3 \pmod 4$ is a prime number and $d>2$ is square-free integer relatively prime to $q$.
	Assume that $F_{1}:=\QQ(\sqrt{q},\sqrt{d},\sqrt{2})$ with $d=rstk$, where $r$, $s$, $t$ and $k$ are prime numbers, and such that $\rank(A(F))=3 $  $($cf. Lemma \ref{lemma1}-$3)$$)$. We have $\rank(A(F_{1}))>3 $.
\end{lemma}

\begin{proof}
	Assume that $r$, $s$, $t$ and $k$ satisfy the conditions  of Lemma \ref{lemma1}-$3-a)$$)$. Let $d=rstk$ such that $\left( \frac{q}{k}\right) =-1$ and $\left( \frac{q}{t}\right) =-1$. 
	Suppose we have $\left( \frac{q}{r}\right) = \left( \frac{q}{s}\right) =1$  such that $[(r \equiv 5 \pmod 8 \text{ and } s \equiv 3  \pmod 4 )$ or $(r \equiv 3\pmod 8 \text{ and } s \equiv  7  \pmod 8 )]$.
		
		If $[r \equiv 5 \pmod 8 \text{ and } s \equiv 3  \pmod 7 ]$, then
		$\rank(A(F_1))=8-1-e_{F_1/F'_{1}}=7-e_{F_1/F'_{1}}>3$, since $-1$ is always a norm in this case.
			Similarly we check that for the remaining cases, we have  $\rank(A(F))> 3$.
\end{proof}

\begin{lemma}
	Let $F:=\QQ(\sqrt{2q},\sqrt{d})$ and $F_{1}:=\QQ(\sqrt{q},\sqrt{d},\sqrt{2})$ where $q\equiv 3 \pmod 4$ is a prime number and $d=\delta rst$ is square-free integer with  $r$, $s$ and $t$ are prime numbers and $\delta \in \{ 1,q \} $   such that $\rank(A(F))=3 $  $($cf. Lemma \ref{lemma3}-$2)$$)$. Then, $\rank(A(F_{1}))=3 $ if and only if one of the following conditions holds:
	 \begin{enumerate}[$1)$]
	  	\item $q \equiv 3 \pmod 8$, $r \equiv s \equiv  3 \pmod 8$,  $t \equiv 7 \pmod 8$ and $\left( \frac{q}{r}\right)=\left( \frac{q}{s}\right)=\left( \frac{q}{t} \right)= -1$,  
	 	
		\item $q \equiv 3 \pmod 8$, $r \equiv s \equiv  5 \pmod 8$, $t \equiv 7 \pmod 8$ and $\left( \frac{q}{r}\right)=\left( \frac{q}{s}\right)=\left( \frac{q}{t} \right)=-1,$
		
		\item  $q \equiv 3 \pmod 8$, $r \equiv s \equiv  3 \pmod 8$, $t \equiv 5 \pmod 8$, $\left( \frac{q}{r}\right)=\left( \frac{q}{s}\right)=-1$ and $ \left( \frac{q}{t} \right)=1,$
		\item  $q \equiv 3 \pmod 8$, $r \equiv s \equiv  5 \pmod 8$, $t \equiv 3, 5 \pmod 8$, $\left( \frac{q}{r}\right)=\left( \frac{q}{s}\right)=-1$ and $ \left( \frac{q}{t} \right)=1,$
		\item  $q \equiv 7 \pmod 8$, $r \equiv s \equiv  3 \pmod 8$, $t \equiv 3, 5 \pmod 8$,$\left( \frac{q}{r}\right)=\left( \frac{q}{s}\right)=-1$ and $\left( \frac{q}{t} \right)=1,$
		\item $q \equiv 7 \pmod 8$, $r \equiv s \equiv  5 \pmod 8$, $t \equiv 3, 5 \pmod 8$, $\left( \frac{q}{r}\right)=\left( \frac{q}{s}\right)=-1$ and $\left( \frac{q}{t} \right)=1,$
		\item $q \equiv 3 \pmod 8$, $r \equiv s \equiv 3 \pmod 8$,  $t \equiv 5 \pmod 8$ and $\left( \frac{q}{r}\right)=\left( \frac{q}{s}\right)=\left( \frac{q}{t}\right)=-1,$
		\item $q \equiv 3 \pmod 8$, $r\equiv s \equiv 3 \pmod 8$,  $t \equiv 7 \pmod 8$, $\left( \frac{q}{r}\right)=\left( \frac{q}{s}\right)=-1$, $\left( \frac{q}{t}\right)=1$ and $\overline{\sqrt{\varepsilon_{2q}}},\overline{\sqrt{\varepsilon_{q}}}  \notin \langle \overline{-1}, \overline{\varepsilon_2} \rangle$,
		\item $q \equiv 3 \pmod 8$, $r \equiv 3 \pmod 8$, $s \equiv t \equiv 5 \pmod 8$ and $\left( \frac{q}{r}\right)=\left( \frac{q}{s}\right)=\left( \frac{q}{t}\right)=-1$,
		\item $q \equiv 3 \pmod 8$, $r \equiv 3 \pmod 8$, $s \equiv 5 \pmod 8$, $t \equiv 7 \pmod 8$, $\left( \frac{q}{r}\right)=\left( \frac{q}{s}\right)=-1$, $\left( \frac{q}{t}\right)=1$ and $\overline{\sqrt{\varepsilon_{q}}}  \notin \langle \overline{\sqrt{\varepsilon_{2q}}},\overline{-1}, \overline{\varepsilon_2} \rangle$,
		\item $q \equiv 3 \pmod 8$, $r \equiv 7 \pmod 8$, $s \equiv 3 \pmod 8$, $t \equiv 5 \pmod 8$, $\left( \frac{q}{r}\right)=1$, $\left( \frac{q}{s}\right)=\left( \frac{q}{t}\right)=-1$ and  $\overline{\sqrt{\varepsilon_{q}}}  \notin \langle \overline{\sqrt{\varepsilon_{2q}}},\overline{-1}, \overline{\varepsilon_2} \rangle$,
		\item $q \equiv 3 \pmod 8$, $r \equiv 7 \pmod 8$, $s \equiv t \equiv 5 \pmod 8$, $\left( \frac{q}{r}\right)=1$, $\left( \frac{q}{s}\right)=\left( \frac{q}{t}\right)=-1$ and $\overline{\sqrt{\varepsilon_{q}}}  \notin \langle \overline{\sqrt{\varepsilon_{2q}}},\overline{-1}, \overline{\varepsilon_2} \rangle$,
		\item $q \equiv 7 \pmod 8$, $r \equiv s \equiv 3 \pmod 8$, $t \equiv 5 \pmod 8$ and $\left( \frac{q}{r}\right)=\left( \frac{q}{s}\right)=\left( \frac{q}{t}\right)=-1$,
		\item $q \equiv 7 \pmod 8$, $r \equiv 3 \pmod 8$, $s \equiv t \equiv 5 \pmod 8$ and $\left( \frac{q}{r}\right)=\left( \frac{q}{s}\right)=\left( \frac{q}{t}\right)=-1$.
	\end{enumerate}
\end{lemma}	

\begin{proof}
	We proceed similarly as in the proof of the previous lemma.
	\begin{enumerate}[$1)$]
		\item Assume that $r$, $s$, $t$ and $q$ satisfy the condition 2)-b) of Lemma \ref{lemma3}, i.e, $\left( \frac{2q}{r} \right)=1$ and $\left( \frac{2q}{s} \right)=-\left( \frac{2q}{t} \right)= 1$ and $[r\equiv s \equiv 3 \pmod 8, r\equiv s \equiv 5 \pmod 8 \text{ or } r\equiv s \equiv 7 \pmod 8]$ and $[(t \equiv 1,7 \pmod 8, \left( \frac{q}{t} \right)=-1), \text{ or } (t \equiv 3,5 \pmod 8, \left( \frac{q}{t} \right)=1)].$ Thus, 2)-b) is equivalent to the fact that we have one of the following conditions:
		
		\begin{enumerate}[$a)$]
			\item $r \equiv s \equiv  7 \pmod 8$, $\left( \frac{q}{r}\right)=\left( \frac{q}{s}\right)=1$ and $t \equiv 1,7 \pmod 8, \left( \frac{q}{t} \right)=-1$,
			\item $r \equiv s \equiv  7 \pmod 8$, $\left( \frac{q}{r}\right)=\left( \frac{q}{s}\right)=1$ and $t \equiv 3,5 \pmod 8, \left( \frac{q}{t} \right)=1$,
			\item $r \equiv s \equiv 3 \pmod 8$,$\left( \frac{q}{r}\right)=\left( \frac{q}{s}\right)=-1$ and $t \equiv 1,7 \pmod 8, \left( \frac{q}{t} \right)=-1$,
			\item $r \equiv s \equiv 3 \pmod 8$,$\left( \frac{q}{r}\right)=\left( \frac{q}{s}\right)=-1$ and $t \equiv 3,5 \pmod 8, \left( \frac{q}{t} \right)=1$,
			\item $r \equiv s \equiv 5 \pmod 8$, $\left( \frac{q}{r}\right)=\left( \frac{q}{s}\right)=-1$ and $t \equiv 1,7 \pmod 8, \left( \frac{q}{t} \right)=-1$,
			\item $r \equiv s \equiv 5 \pmod 8$, $\left( \frac{q}{r}\right)=\left( \frac{q}{s}\right)=-1$ and $t \equiv 3,5 \pmod 8, \left( \frac{q}{t} \right)=1$.
		\end{enumerate}
		Assume that $r$, $s$ and $t$ satisfy $a)$ or $b)$, by the proof of  \cite[Lemma 3.5]{2MCD2026}, we conclude that $\rank(A(F_{1}))>3$.
		
		Assume that $r$,$s$ and $t$ satisfy $c)$ and $q \equiv 3 \pmod 8$, then:
		\begin{enumerate}[\rm $\bullet$]
		\item  If $t \equiv 1 \pmod 8$, then $\rank(A(F_{1}))=t_{F_{1}/F'_{1}}-1-e_{F_{1}/F'_{1}}=6-1-e_{F_{1}/F'_{1}}= 5-e_{F_{1}/F'_{1}}=5-1=4$.

	\item  If $t \equiv 7 \pmod 8$, then $\rank(A(F_{1}))=t_{F_{1}/F'_{1}}-1-e_{F_{1}/F'_{1}}=6-1-e_{F_{1}/F'_{1}}= 5-e_{F_{1}/F'_{1}}=5-2=3$. 
		
	\end{enumerate}
		We proceed similarly for the rest.
		\item Assume that $r$, $s$, $t$ and $q$ satisfy the condition 2)-c) of Lemma \ref{lemma3}, i.e, $\left( \frac{2q}{r} \right)=\left( \frac{2q}{s} \right)=\left( \frac{2q}{t} \right)= 1$ and $\left( \frac{-1}{r} \right)=\left( \frac{2}{s} \right)=-1 \text{ and } \left( \frac{-1}{t} \right) \neq \left( \frac{2}{t} \right)$.
		We distinguish the following cases for $r$, $s$ and $t$:
		\begin{enumerate}[\rm $\bullet$]
			\item $[ r\equiv 3 \pmod 8$, $\left( \frac{q}{r}\right)=-1]$ or $[ r\equiv 7 \pmod 8$, $\left( \frac{q}{r}\right)=1]$, 
			\item $ s \equiv 3,5 \pmod 8$, $\left( \frac{q}{s}\right)=-1$,
			\item $[ t\equiv 5 \pmod 8$, $\left( \frac{q}{t}\right)=-1]$ or $[ t\equiv 7 \pmod 8$, $\left( \frac{q}{t}\right)=1]$.
		\end{enumerate}
		\ding{229}	Suppose $q\equiv 3 \pmod 8$, then we summarize the different results in the following table.
		\begin{center}
			\begin{tabular}{p{5cm}|p{7cm}|p{5cm}}
				\textbf{Case 1: for the prime $r$} &
				\textbf{Case 2: for the prime $s$} &
				\textbf{Case 3: for the prime $t$} \\
				
				$\begin{array}{l}
					\text{if } r \equiv 3 \pmod{8}\\[0.2cm]
					\left( \frac{-1,d}{\mathcal{R}} \right)=1\\[0.2cm]
					\left( \frac{\varepsilon_2,d}{\mathcal{R}} \right)=-1\\[0.2cm]
					\left( \frac{\sqrt{\varepsilon_q},d}{\mathcal{R}} \right)=1\\[0.2cm]
					\left( \frac{\sqrt{\varepsilon_{2q}},d}{\mathcal{R}} \right)=1\\[0.2cm]
					\text{if } r \equiv 7 \pmod{8}\\[0.2cm]
					\left( \frac{-1,d}{\mathcal{R}} \right)=-1\\[0.2cm]
					\left( \frac{-1,d}{\mathcal{R}'} \right)=-1\\[0.2cm]
					\left( \frac{\varepsilon_2,d}{\mathcal{R}} \right)=1\\[0.2cm]
					\left( \frac{\varepsilon_2,d}{\mathcal{R}'} \right)=-1\\[0.2cm]
					
				\end{array}$
				&
				$\begin{array}{l}
					\left( \frac{-1,d}{\mathcal{S}} \right)=1\\[0.2cm]
					\left( \frac{\varepsilon_2,d}{\mathcal{S}} \right)=
					\begin{cases}
						-1, s \equiv 3 \pmod{8},\\
						1, s \equiv 5 \pmod{8}
					\end{cases}\\[0.2cm]
					\left( \frac{\sqrt{\varepsilon_q},d}{\mathcal{S}} \right)=1\\[0.2cm]
					\left( \frac{\sqrt{\varepsilon_{2q}},d}{\mathcal{S}} \right)=
					\begin{cases}
						1, s \equiv 3 \pmod{8},\\
						-1, s \equiv 5 \pmod{8}
					\end{cases}
				\end{array}$
				&
				$\begin{array}{l}
					\text{if } t \equiv 5 \pmod{8}\\[0.2cm]
					\left( \frac{-1,d}{\mathcal{T}} \right)=1\\[0.2cm]
					\left( \frac{\varepsilon_2,d}{\mathcal{T}} \right)=-1\\[0.2cm]
					\left( \frac{\sqrt{\varepsilon_q},d}{\mathcal{T}} \right)=1\\[0.2cm]
					\left( \frac{\sqrt{\varepsilon_{2q}},d}{\mathcal{T}} \right)=1\\[0.2cm]
					\text{if } t \equiv 7 \pmod{8}\\[0.2cm]
					\left( \frac{-1,d}{\mathcal{T}} \right)=-1\\[0.2cm]
					\left( \frac{-1,d}{\mathcal{T}'} \right)=-1\\[0.2cm]
					\left( \frac{\varepsilon_2,d}{\mathcal{T}} \right)=1\\[0.2cm]
					\left( \frac{\varepsilon_2,d}{\mathcal{T}'} \right)=-1\\[0.2cm]
					
				\end{array}$
			\end{tabular}
		\end{center}
	 	Checking all cases yields the stated result. 
	\end{enumerate}      
\end{proof}

\begin{lemma}
	Let $F:=\QQ(\sqrt{2q},\sqrt{d})$ and $F_{1}:=\QQ(\sqrt{q},\sqrt{d},\sqrt{2})$ where $q\equiv 3 \pmod 4$ is a prime number and $d=\delta rstk$ is square-free integer with $r$, $s$, $t$ and $k$ are prime numbers and $\delta \in \{ 1,q \} $   such that $rank(A(F))=3 $ $($cf. Lemma \ref{lemma3}-$3)$$)$.	We have $\rank(A(F_{1}))>3 $. 
\end{lemma}

\begin{proof}
	\begin{enumerate}[$1)$]
		\item Assume that $r$, $s$, $t$ and $k$ satisfy the condition 3)-a) of Lemma \ref{lemma3}, i.e, $\left( \frac{2q}{r} \right)=\left( \frac{2q}{s} \right)= \left( \frac{2q}{t} \right)=\left( \frac{2q}{k} \right)=-1$,
		then  $\rank(A(F_{1}))=t_{F_{1}/F'_{1}}-1-e_{F_{1}/F'_{1}}=8-1-e_{F_{1}/F'_{1}}= 7-e_{F_{1}/F'_{1}}$, and it is  equal to $3$ if and only if $e_{F_{1}/F'_{1}}=4$, which is not the case since $-1$ is always a norm in these cases.
		\item Assume that $r$, $s$, $t$ and $k$ satisfy the condition 3)-b) of Lemma \ref{lemma3}, i.e, $\left( \frac{2q}{r} \right)=-\left( \frac{2q}{s} \right)=- \left( \frac{2q}{t} \right)=-\left( \frac{2q}{k} \right)=1$ and $[r\equiv  3 \pmod 4 \text{ or } r \equiv 5 \pmod 8 ]$, we check that $\rank(A(F_{1}))>3$.

		 We proceed similarly for the rest.
	\end{enumerate}
\end{proof}

	

\begin{thebibliography}{11}
		
		
		
			\bibitem{Azmouh2-rank} A. Azizi,   et A. Mouhib,	Le 2-rang du groupe de classes de certains corps biquadratiques et applications, International Journal of Mathematics, 15 (2004),   pp. 169-182. 
		
	\bibitem{2MCD2026}
	M. M. Chems-Eddin, Greenberg's conjecture and Iwasawa module of Real biquadratic fields I, Journal of Number Theory, 281 (2026), 224--266.
	
	\bibitem{chemskatharina}
	M. M. Chems-Eddin and K. M\"uller, \textit{$2$-Class groups of cyclotomic towers of imaginary biquadratic fields and applications}, Int. J. Number Theory 17 (2021), 1781--1793.
	
	
	\bibitem{Elmahi1}
	A. El Mahi, On the structure of the Iwasawa module for $\mathbb{Z}_2$-extensions of certain real biquadratic fields, Acta Math. Hungar. 174 (2024), 49--61.
	
	\bibitem{Elmahi2}
	A. El Mahi and M. Ziane, On Greenberg's conjecture for certain real biquadratic fields, Period. Math. Hungar. 88 (2024), 475--483.
	
	\bibitem{Elmahi3}
	A. El Mahi and M. Ziane, The Iwasawa invariant $\mu$ vanishes for $\mathbb{Z}_2$-extensions of certain real biquadratic fields, Acta Math. Hungar. 165 (2021), 146--155.
	
	\bibitem{fukudaa}
	T. Fukuda, Remarks on $\mathbb{Z}_p$-extensions of number fields, Proc. Jpn. Acad. Ser. A Math. Sci. 70 (1994), 264--266.
	
	\bibitem{iwasawa59}
	K. Iwasawa, On $\Gamma$-extensions of algebraic number fields, \textit{Bull. Amer. Math. Soc.} 65 (1959), 183--226.
	
	\bibitem{kuruda}
	S. Kuroda, \"Uber die Klassenzahlen algebraischer Zahlk\"orper, Nagoya Math. J. 1 (1950), 1--10.
	
	\bibitem{8}
	H. Laxmi and A. Saikia, $\mathbb{Z}_2$-extension of real quadratic fields with $\mathbb{Z}/2\mathbb{Z}$ as $2$-class group at each layer, Ramanujan J. 64 (2024), 1285--1301. https://doi.org/10.1007/s11139-024-00869-8.
	
	\bibitem{Laxmi}
	H. Laxmi and A. Saikia, Stability of $2$-class groups in the $\mathbb{Z}_2$-extension of certain real biquadratic fields, preprint (2025), https://arxiv.org/abs/2501.12782.
	
	\bibitem{mizusawa5}
	Y. Mizusawa, On metabelian $2$-class field towers over $\mathbb{Z}_2$-extensions of real quadratic fields, Can. Math. Bull. 65 (2022), 795--805.
	
	\bibitem{MR2024}
	A. Mouhib and S. Rouas, On unramified Galois $2$-groups over $\mathbb{Z}_2$-extensions of some imaginary biquadratic number fields, Acta Math. Hungar. 172 (2024), 481--491.
	
	\bibitem{washington1997introduction}
	L. C. Washington, \textit{Introduction to Cyclotomic Fields}, Graduate Texts in Mathematics, Vol. 83, Springer-Verlag, New York, 1982.
	
	\bibitem{yue}
	Q. Yue, The generalized R\'edei matrix, Math. Z. 261 (2009), 23--37.
		
	 
		
		
	\end{thebibliography}
\end{document}